\documentclass[12pt]{amsart}
\usepackage{amsthm}
\usepackage{amssymb}
\usepackage{amsmath}
\usepackage{color}
\theoremstyle{plain}
\newtheorem{proposition}{Proposition}
\newtheorem{theorem}[proposition]{Theorem}
\newtheorem{lemma}[proposition]{Lemma}
\newtheorem{corollary}[proposition]{Corollary}

\theoremstyle{definition}

\theoremstyle{definition}
\newtheorem{example}[proposition]{Example}

\newtheorem{remark}[proposition]{Remark}

\numberwithin{equation}{section}

\numberwithin{proposition}{section}

\gdef\myletter{}

\let\savetheequation\theequation
\def\theequation{\savetheequation\myletter}

\usepackage{color}

\def\red{\color{red}}

\newcommand{\CC}{{\mathbb C}}

\newcommand{\RR}{{\mathbb R}}
\newcommand{\ZZ}{{\mathbb Z}}

\newcommand{\NN}{{\mathbb N}}

\newcommand{\T}{(\Bbb{C}^*)^d}

\renewcommand{\date}{\today}
\def \bar{\overline}
\def \hat{\widehat}

\begin{document}


\title[Siciak-Zaharjuta]{\bf Pluripotential Theory and Convex Bodies: A Siciak-Zaharjuta theorem}

\author{T. Bayraktar,* S. Hussung, N. Levenberg** and M. Perera}{\thanks{*Supported by The Science Academy BAGEP, **Supported by Simons Foundation grant No. 354549}}

\maketitle
\begin{abstract} We work in the setting of weighted pluripotential theory arising from polynomials associated to a convex body $P$ in $(\RR^+)^d$. We define the {\it logarithmic indicator function} on $\CC^d$:
$$H_P(z):=\sup_{ J\in P} \log |z^{ J}|:=\sup_{ J\in P} \log[|z_1|^{ j_1}\cdots |z_d|^{ j_d}]$$ and an associated class of plurisubharmonic (psh) functions:
$$L_P:=\{u\in PSH(\CC^d): u(z)- H_P(z) =0(1), \ |z| \to \infty \}.$$ We first show that $L_P$ is not closed under standard smoothing operations. However, utilizing a continuous regularization due to Ferrier which preserves $L_P$, we prove a general Siciak-Zaharjuta type-result in our $P-$setting: the weighted $P-$extremal function
$$V_{P,K,Q}(z):=\sup \{u(z):u\in L_P, \ u\leq Q \ \hbox{on} \ K\}$$
associated to a compact set $K$ and an admissible weight $Q$ on $K$ can be obtained using the subclass of $L_P$ arising from functions of the form $\frac{1}{deg_P(p)}\log |p|$ (appropriately normalized).
\end{abstract}

\section{Introduction} A fundamental result in pluripotential theory is that the extremal plurisubharmonic function
$$V_K(z):= \sup \{u(z):u\in L(\CC^d), \ u\leq 0 \ \hbox{on} \ K\}$$
associated to a compact set $K\subset \CC^d$, where $L(\CC^d)$ is the usual Lelong class of all plurisubharmonic (psh) functions $u$ on $\CC^d$ with the property that $u(z) - \log |z| = 0(1)$ as $|z| \to \infty$, may be obtained from the subclass of $L(\CC^d)$ arising from polynomials:
$$V_K(z)=\max[0,\sup\{\frac{1}{deg(p)}\log |p(z)|:p \ \hbox{polynomial}, \ ||p||_K\leq 1\}].$$
More generally, given an admissible weight function $Q$ on $K$ ($Q$ is lowersemicontinuous and $\{z\in K:Q(z)<\infty\}$ is not pluripolar), 
$$V_{K,Q}(z):= \sup \{u(z):u\in L(\CC^d), \ u\leq Q \ \hbox{on} \ K\}$$
$$=\max[0,\sup\{\frac{1}{deg(p)}\log |p(z)|:p \ \hbox{polynomial}, \ ||pe^{-deg(p)\cdot Q}||_K\leq 1\}].$$
We refer to this as a {\it Siciak-Zaharjuta type} result. Standard proofs often reduce to a sufficiently regular case by regularization; i.e., convolving with a smooth bump function. 

In recent papers, a (weighted) pluripotential theory associated to a convex body $P$ in $(\RR^+)^d$ has been developed. Let $\RR^+=[0,\infty)$ and fix a convex body $P\subset (\RR^+)^d$ ($P$ is compact, convex and $P^o\not = \emptyset$). An important example is when $P$ is a non-degenerate convex polytope, i.e., the convex hull of a finite subset of $(\RR^+)^d$ with nonempty interior. Associated with $P$ we consider the finite-dimensional polynomial spaces 
$$Poly(nP):=\{p(z)=\sum_{J\in nP\cap (\ZZ^+)^d}c_J z^J: c_J \in \CC\}$$
for $n=1,2,...$ where $z^J=z_1^{j_1}\cdots z_d^{j_d}$ for $J=(j_1,...,j_d)$. For $P=\Sigma$ where 
$$\Sigma:=\{(x_1,...,x_d)\in \RR^d: 0\leq x_i \leq 1, \ \sum_{i=1}^d x_i \leq 1\},$$ we have $Poly(n\Sigma)$ is the usual space of holomorphic polynomials of degree at most $n$ in $\CC^d$. For a nonconstant polynomial $p$ we define 
\begin{equation}\label{degp} \deg_P(p)=\min\{ n\in\NN \colon p\in Poly(nP)\}.\end{equation}

We define the {\it logarithmic indicator function} of $P$ on $\CC^d$
$$H_P(z):=\sup_{ J\in P} \log |z^{ J}|:=\sup_{ J\in P} \log[|z_1|^{ j_1}\cdots |z_d|^{ j_d}].$$
Note that $H_P(z_1,...,z_d)=H_P(|z_1|,...,|z_d|)$. As in \cite{BBL}, \cite{BBLL}, \cite{BosLev}, we make the assumption on $P$ that 
\begin{equation}\label{sigmainkp} \Sigma \subset kP \ \hbox{for some} \ k\in \ZZ^+.  \end{equation}
In particular, $0\in P$. Under this hypothesis, we have
\begin{equation} \label{sigmainkp2}  H_P(z)\geq \frac{1}{k}\max_{j=1,...,d}\log^+ |z_j| \end{equation}
where $\log^+ |z_j| =\max[0,\log|z_j|]$. We use $H_P$ to define generalizations of the Lelong classes $L(\CC^d)$ and 
$$L^+(\CC^d)=\{u\in L(\CC^d): u(z)\geq \max_{j=1,...,d} \log^+ |z_j| + C_u\}$$
where $C_u$ is a constant depending on $u$. Define
$$L_P=L_P(\CC^d):= \{u\in PSH(\CC^d): u(z)- H_P(z) =0(1), \ |z| \to \infty \},$$ and 
$$L_{P,+}=L_{P,+}(\CC^d)=\{u\in L_P(\CC^d): u(z)\geq H_P(z) + C_u\}.$$
For $p\in Poly(nP), \ n\geq 1$ we have $\frac{1}{n}\log |p|\in L_P$; also each $u\in L_{P,+}$ is locally bounded in $\CC^d$. Note $L_{\Sigma} = L(\CC^d)$ and $L_{\Sigma,+} = L^+(\CC^d)$. 

Given $E\subset \CC^d$, the {\it $P-$extremal function of $E$} is given by $V^*_{P,E}(z):=\limsup_{\zeta \to z}V_{P,E}(\zeta)$ where
$$V_{P,E}(z):=\sup \{u(z):u\in L_P(\CC^d), \ u\leq 0 \ \hbox{on} \ E\}.$$
Introducing weights, let $K\subset \CC^d$ be closed and let $w:K\to \RR^+$ be a nonnegative, uppersemicontinuous function with $\{z\in K:w(z)>0\}$ nonpluripolar. Letting $Q:= -\log w$, if $K$ is unbounded, we additionally require that 
\begin{equation}\label{unbweight} \liminf_{|z|\to \infty, \ z\in K} [Q(z)- H_P(z)]=+\infty.\end{equation}
Define the {\it weighted $P-$extremal function} $$V^*_{P,K,Q}(z):=\limsup_{\zeta \to z}V_{P,K,Q}(\zeta)$$ where
$$V_{P,K,Q}(z):=\sup \{u(z):u\in L_P(\CC^d), \ u\leq Q \ \hbox{on} \ K\}. $$
If $Q=0$ we simply write $V_{P,K,Q}=V_{P,K}$ as above. For $P=\Sigma$, 
$$V_{\Sigma,K,Q}(z)=V_{K,Q}(z):= \sup \{u(z):u\in L(\CC^d), \ u\leq Q \ \hbox{on} \ K\} $$
is the usual weighed extremal function.

A version of a Siciak-Zaharjuta type result has been given in \cite{TU} in the case where it is assumed that $V_{P,K,Q}$ is continuous. Here we give a complete proof of the general version:

\begin{theorem} \label{maint}  Let $P\subset (\RR^+)^d$ be a convex body, $K\subset \CC^d$ closed, and $w=e^{-Q}$ an admissible weight on $K$. Then 
$$V_{P,K,Q} =\lim_{n\to \infty} \frac{1}{n} \log \Phi_n=\lim_{n\to \infty} \frac{1}{n} \log \Phi_{n,P,K,Q}$$
pointwise on $\CC^d$ where
\begin{equation} \label{phin} \Phi_n(z):= \sup \{|p_n(z)|: p_n\in Poly(nP),  \ \max_{\zeta \in K} |p_n(\zeta )e^{-nQ(\zeta)}|\leq 1\}.\end{equation}
If $V_{P,K,Q}$ is continuous, we have local uniform convergence on $\CC^d$.

\end{theorem}

In the next section, we show that standard convolution does {\it not} necessarily preserve the $L_P$ classes. Thus the transition from the Siciak-Zaharjuta type result for $V_{P,K,Q}$ continuous to general $V_{P,K,Q}$ is not immediate. In section 3, we recall the Ferrier regularization procedure from \cite{F} and show that it does preserve the $L_P$ classes. Then in sections 4 and 5 we present a complete proof of Theorem \ref{maint} together with remarks on regularity of $P-$extremal functions.

\section{Approximation by convolution} We fix a standard smoothing kernel 
                   \begin{equation}\label{chi} \chi(z)=\chi(z_1,...,z_d)=\chi(|z_1|,...,|z_d|)\end{equation} with $0\leq \chi \leq 1$ and support in the unit polydisk satisfying $\int \chi dV=1$ where $dV$ is the standard volume form on $\CC^d$. Let $\chi_{1/j}(z) =j^{2d}\chi(jz)$. For which $P$ does $u\in L_P$ imply $u_j :=u*\chi_{1/j}\in L_P$ for $j$ sufficiently large? To determine this, it clearly suffices to consider $u=H_P$. Thus we write $u_j(z):=   (H_P*\chi_{1/j}) (z)$. 
                                    
                  For general $P$ we  know that $u_j \geq H_P$; $u_j \downarrow H_P$ pointwise on $\CC^d$ and uniformly on compact sets. Thus if $u_j \in L_P$ then, in fact, $u_j  \in L_{P,+}$.
                  
                  Fix $\delta >0$ so that 
                  \begin{enumerate}
                  \item for $j\geq j_0(\delta)$ we have $u_j(z)\leq H_P(z) +\delta$ if $|z_1|,..., |z_d|\leq {1 \over \delta}$ and 
                  \item $u_j(z)\leq H_P(z)+C(\delta)$ if $|z_1|,..., |z_d| \geq \delta$ for all $j$ where $C(\delta)$ depends only on $\delta$. \end{enumerate}
                  
               \noindent Property (1) follows from the local uniform convergence. For (2), 
                \begin{equation}\label{dzj} u_j(z) \leq \max_{D(z,1/j)} H_P\leq  H_P(|z_1|+1/j,...,|z_d|+1/j)\end{equation}
                  where $D(z,1/j)$ is the polydisk of polyradius $(1/j,...,1/j)$ centered at $z$. Then for $|z_k|>\delta$, 
              $$\log (|z_k|+1/j)= \log |z_k| + \log (1+ {1\over j|z_k|})$$
              $$\leq \log |z_k|+ \log (1+{1\over j\delta})\leq \log |z_k|+ \log (1+{1\over \delta})$$ and since
              $$H_P(z) = \sup_{(j_1,...,j_d)\in P} \log[|z_1|^{j_1}\cdots |z_d|^{j_d}]=\sup_{(j_1,...,j_d)\in P} \sum_{k=1}^d j_k\log |z_k|,$$
              for $|z_1|,..., |z_d| \geq \delta$ we have
              $$\max_{D(z,1/j)} H_P\leq \sup_{(j_1,...,j_d)\in P} \sum_{k=1}^d j_k\log (|z_k|+1/j)$$
              $$\leq \sup_{(j_1,...,j_d)\in P} \sum_{k=1}^d j_k\log |z_k|+ \log (1+{1\over \delta}) \cdot \sup_{(j_1,...,j_d)\in P} \sum_{k=1}^d j_k$$
              which gives (2).              
              
              For simplicity we work in $\CC^2$ with variables $(z_1,z_2)$. From the above calculations, we see that, given $\delta>0$, fixing $j\geq j_0(\delta)$, to show $u_j  \in L_P$ it suffices to show there is a constant $A(\delta)$ depending only on $\delta$ such that for $(z_1,z_2)$ with $|z_1|<\delta$ and $|z_2|>1/\delta$ and for $(z_1,z_2)$ with $|z_1|>1/\delta$ and $|z_2|<\delta$, we have 
              \begin{equation}\label{three} u_j(z_1,z_2)\leq H_P(z_1,z_2)+ A(\delta). \end{equation}
              
              \begin{proposition} \label{propA} If there exists $\delta >0$ so that $H_P(z_1,z_2)\geq H_P(\delta,z_2)$ for $|z_1|< \delta$ and $|z_2|> 1/\delta$ as well as $H_P(z_1,z_2)\geq  H_P(z_1,\delta)$ for $|z_2|< \delta$ and $|z_1|> 1/\delta$ then $u_j = H_P*\chi_{1/j} \in L_P$ for $j$ sufficiently large.\end{proposition} 
              
       \begin{proof} We need to prove (\ref{three}); to do this it suffices to show 
              $$u_j(z_1,z_2)\leq H_P(\delta,z_2)+ A(\delta), \ |z_1|<\delta \ \hbox{and} \ |z_2|>1/\delta \eqno(A)$$
              and
              $$u_j(z_1,z_2)\leq H_P(z_1,\delta)+ A(\delta), \ |z_1|>1/\delta \ \hbox{and} \ |z_2|<\delta. \eqno(B)$$ 
             We verify (A); (B) is the same. To verify (A), for such $(z_1,z_2)$, from (\ref{dzj}), we need the appropriate upper bound on
              $$\sup_{(j_1,j_2)\in P} [j_1 \log (|z_1|+1/j)+j_2\log (|z_2|+1/j)].$$
              Now 
              $$j_1 \log (|z_1|+1/j)+j_2\log (|z_2|+1/j)\leq j_1 \log (\delta+1/j)+j_2\log (|z_2|+1/j)$$
              $$= j_1 [\log \delta+\log (1+{1\over \delta j})]+j_2[\log |z_2|+\log (1+{1\over |z_2|j})]$$
              $$\leq j_1 [\log \delta+\log (1+{1\over \delta j})]+j_2[\log |z_2|+\log (1+{\delta \over j})]$$
              $$\leq j_1 [\log \delta+\log (1+{1\over \delta })]+j_2[\log |z_2|+\log (1+\delta )].$$
              Thus
              $$u_j(z_1,z_2)\leq \sup_{(j_1,j_2)\in P} [j_1 \log (|z_1|+1/j)+j_2\log (|z_2|+1/j)]$$
              $$\leq \sup_{(j_1,j_2)\in P} \bigl(j_1 [\log \delta+\log (1+{1\over \delta })]+j_2[\log |z_2|+\log (1+\delta )]\bigr)$$
              $$\leq \sup_{(j_1,j_2)\in P}[j_1 \log \delta + j_2\log |z_2|] + A(\delta)$$
              where 
              $$A(\delta)=\sup_{(j_1,j_2)\in P}[j_1 \log (1+{1\over \delta })+j_2\log (1+\delta )].$$
              \end{proof}
              
              We call a convex body $P\subset (\RR^+)^d$ a {\it lower set} if for each $n=1,2,...$, whenever $(j_1,...,j_d) \in nP\cap (\ZZ^+)^d$ we have $(k_1,...,k_d) \in nP\cap (\ZZ^+)^d$ for all $k_l\leq j_l, \ l=1,...,d$. Clearly $H_P$ for such $P\subset  (\RR^+)^2$ satisfy the hypotheses of Proposition \ref{propA}.
              
              \begin{corollary} If $P \subset (\RR^+)^2$ is a lower set, then $H_P*\chi_{1/j} \in L_P$ for $j$ sufficiently large.
          \end{corollary} 
          
          Indeed, it appears this condition is {\it necessary} for $H_P*\chi_{1/j} \in L_P$ as the following explicit example indicates. 
          
          \begin{example} Let $P$ be the quadrilateral with vertices $(0, 0), (1, 0), (0, 1)$ and $(1, 2)$. This $P$ is not a lower set. We show that for $\epsilon >0$ sufficiently small, $H_P*\chi_{\epsilon} \not\in L_P$. Here, 
$$H_P(z_1,z_2)=\max[ 0,\log |z_1|, \log |z_2|, \log|z_1|+2\log|z_2|].$$
Consider the regions
$$A:=\{(z_1,z_2): |z_1z_2|< 1, \ |z_2|>1\}$$
and
$$B:=\{(z_1,z_2):  |z_1z_2|>1, \  |z_2|>1\}.$$
In $A$, $H_P(z_1,z_2)=\log |z_2|$ while in $B$ we have $H_P(z_1,z_2)=\log|z_1|+2\log|z_2|=\log |z_2| +\log |z_1z_2|$. 
Fixing $\epsilon >0$, we take any large $C$. We claim we can find a point $(x_C,y_C)$ at which 
$$H_P*\chi_{\epsilon}(x_C,y_C)-H_P(x_C,y_C) >C.$$
If $0<|x|<1/|y|<1$ then $(x,y)\in A$. For such $(x,y)$, let 
$$D_{\epsilon}(x,y):=\{(z_1,z_2): |z_1-x|, |z_2-y| <\epsilon\}$$
and 
$$S_{\epsilon}(x,y):=\{(z_1,z_2)\in D_{\epsilon}(x,y): |z_1|>\epsilon/2 \}.$$
We first choose $y_C$ with $|y_C|$ sufficiently large so that for any choice of $x_C$ with $|x_C|<1/|y_C|$ -- so that $Z:=(x_C,y_C)\in A$ -- we have
$$|y_C| > \max[\frac{4}{\epsilon},  \frac{2}{\epsilon}e^{2C/ A_{\epsilon}}]+ \epsilon$$
where $A_{\epsilon} :=\int_{\tilde S_{\epsilon}}\chi_{\epsilon} dV$ and 
$$\tilde S_{\epsilon}:=  \{(z_1,z_2): \epsilon/2 \leq |z_1| \leq \epsilon, \ |z_2| < \epsilon \}.$$
Note then that $|x_C|<\epsilon/4$ and that $S_{\epsilon}(x_C,y_C)$ contains the set of points 
$$\{(z_1,z_2): \epsilon/2 \leq |z_1| \leq \epsilon, |z_2 -y_C| < \epsilon \}$$
which is a translation of $\tilde S_{\epsilon}$, centered at $(0,0)$, to $(0,y_C)$. The choice of $y_C$ insures that $S_{\epsilon}(x_C,y_C) \subset B$ and 
$$\hbox{for}  \ (\zeta_1,\zeta_2)\in S_{\epsilon}(x_C,y_C) \ \hbox{we have} \ |\zeta_1\zeta_2|>1 \ \hbox{and} \ |\zeta_2| \geq  \frac{2}{\epsilon}e^{2C/A_{\epsilon}}.$$ Writing $\zeta:=(\zeta_1,\zeta_2)$, we have
$$H_P*\chi_{\epsilon}(Z)-H_P(Z)=\int_{D_{\epsilon}(x_C,y_C)}H_P(\zeta)\chi_{\epsilon}(Z-\zeta)dV(\zeta)-H_P(Z)$$
$$=\int_{D_{\epsilon}(x_C,y_C)} [\log |\zeta_2|+ \log^+|\zeta_1\zeta_2|]\chi_{\epsilon}(Z-\zeta)dV(\zeta) -\log |y_C|$$
$$=\int_{D_{\epsilon}(x_C,y_C)} \log^+|\zeta_1\zeta_2|\chi_{\epsilon}(Z-\zeta)dV(\zeta)$$
since $\zeta \to \log |\zeta_2|$ is harmonic on $D_{\epsilon}(x_C,y_C)$ and $\int_{D_{\epsilon}}\chi_{\epsilon}(Z-\zeta)dV(\zeta)=1$. But then
$$\int_{D_{\epsilon}(x_C,y_C)} \log^+|\zeta_1\zeta_2|\chi_{\epsilon}(Z-\zeta)dV(\zeta)\geq \int_{S_{\epsilon}(x_C,y_C)} \log|\zeta_1\zeta_2|\chi_{\epsilon}(Z-\zeta)dV(\zeta)$$
$$\geq \int_{S_{\epsilon}(x_C,y_C)}\log [\frac{\epsilon}{2} \frac{2}{\epsilon}e^{2C/A_{\epsilon}}]\chi_{\epsilon}(Z-\zeta)dV(\zeta) $$
$$ = \log e^{2C/ A_{\epsilon}} \int_{S_{\epsilon}(x_C,y_C)}\chi_{\epsilon}(Z-\zeta)dV(\zeta)$$
$$ \geq \frac{2C}{A_{\epsilon}} \int_{\tilde S_{\epsilon}}\chi_{\epsilon}(\zeta)dV(\zeta) = 2C.$$
\end{example}

              We will use this standard regularization procedure in the proof of Theorem \ref{maint} but in our application we only utilize the monotonicity property $u_j \downarrow u$. In the next section, we discuss an alternate regularization procedure which always preserves $L_P$ classes and which will be needed to complete the proof of Theorem \ref{maint}.
              
              \section{Ferrier approximation} 
              
              We can do a global approximation of $u\in L_{P}$ from above by continuous $u_t \in L_{P}$ following the proof of Proposition 1.3 in \cite{Si} which itself is an adaptation of \cite{F}. 

\begin{proposition}\label{ga} Let $u\in  L_{P}$. For $t>0$, define 
\begin{equation}\label{ut} u_t(x):=-\log \bigl[\inf_{y\in \CC^d}\{e^{-u(y)}+\frac{1}{t}|y-x|\}\bigr].\end{equation}
Then for $t>0$ sufficiently small, $u_t \in L_{P}\cap C(\CC^d)$ and $u_t\downarrow u$ on $\CC^d$.
\end{proposition}

\begin{proof} The continuity of $u_t$ follows from continuity of $\delta_t(x):=e^{-u_t(x)}$ which follows from the estimate
$$\delta_t(x)-\delta_t(y)=e^{-u_t(x)}-e^{-u_t(y)}\leq \frac{1}{t}|x-y|.$$
Note that $\delta_t\uparrow$ so $u_t \downarrow$. Since $\inf_{y\in \CC^d}\{e^{-u(y)}+\frac{1}{t}|y-x|\}\leq e^{-u(x)}$, we have $u_t(x)\geq u(x)$. To show that $u_t\downarrow u$ on $\CC^d$, fix $x\in \CC^d$. By adding a constant we may assume $u(x)=0$. Given $\delta >0$, we want to show there exists $t(\delta)>0$ such that $u_t(x) < \delta$ for $t< t(\delta)$. Thus we want 
\begin{equation}\label{uttou} \inf_{y\in \CC^d}\{e^{-u(y)}+\frac{1}{t}|y-x|\}> e^{-\delta} \ \hbox{for} \ t< t(\delta).\end{equation}
Since $e^{-u}$ is lowersemicontinuous and $e^{-u(x)}=1> e^{-\delta}$, we can find $\epsilon >0$ so that 
$$e^{-u(y)} > e^{-\delta} \ \hbox{for} \ |y-x|< \epsilon.$$
For such $y$, we have $e^{-u(y)}+\frac{1}{t}|y-x|> e^{-\delta}$ for any $t>0$. Choosing $t(\delta)>0$ so that 
$t(\delta) < \epsilon e^{\delta}$ achieves (\ref{uttou}).

The proof that $u_t$ is psh follows \cite{F}; for the reader's convenience we include this in an appendix. Given this, we are left to show $u_t\in L_P$ for $t>0$ sufficiently small. It clearly suffices to show this for $u=H_P$. Thus, let                          $$u_t(x):=-\log [ \inf_y \{e^{-H_P(y)}+{1\over t} |y-x|\}].$$
For $t>0$ sufficiently small, we want to show there exists $R>>1$ and $0<C<1$, both depending only on $t$ and $P$, so that for each $x\in \CC^d$ with $e^{H_P(x)}>R$ we have 
                         $$e^{-u_t(x)}\geq C e^{-H_P(x)}.$$
                         Unwinding this last inequality, we require
                         $$e^{-H_P(y)}+{1\over t} |y-x|\geq C e^{-H_P(x)} \ \hbox{for all} \ y\in \CC^d.$$
                         This is the same as
                         \begin{equation}\label{one}{Ce^{H_P(y)}-e^{H_P(x)}\over e^{H_P(x)}e^{H_P(y)}}\leq {1\over t} |y-x|  \ \hbox{for all} \ y\in \CC^d.\end{equation}
                                                  
                         Fix $x$ and fix $C$ with $0<C<1$. For any $y$ with $e^{H_P(y)}\leq {1\over C}e^{H_P(x)}$, (\ref{one}) is clearly satisfied. If $e^{H_P(y)}\geq {1\over C}e^{H_P(x)}$, since
$${Ce^{H_P(y)}-e^{H_P(x)}\over e^{H_P(x)}e^{H_P(y)}}\leq {Ce^{H_P(y)}\over e^{H_P(x)}e^{H_P(y)}}={C\over e^{H_P(x)}},$$
we would like to have
 \begin{equation}\label{two}{C\over e^{H_P(x)}}\leq {1\over t} |y-x|.\end{equation}
To estimate $|y-x|$, note that $x$ lies on the set 
$$L_x:=\{z: e^{H_P(z)}=e^{H_P(x)}\}$$
while $y$ lies outside the larger level set
$$L_{C,x}:=\{z: e^{H_P(z)}={1\over C}e^{H_P(x)}\}.$$
Thus 
$$|y-x|\geq dist(L_x,L_{C,x})$$
and it suffices, for (\ref{two}), to have
\begin{equation}\label{tree}{C\over e^{H_P(x)}}\leq {1\over t}dist(L_x,L_{C,x}). \end{equation}
Note that $L_x$ depends only on $x$ (and $P$) while $L_{C,x}$ depends only on $C$ and $x$ (and $P$). But for any fixed $C$ with $0<C<1$, $dist(L_x,L_{C,x})$ is bounded below by a positive constant as $H_P(x)\to \infty$ for a convex body $P\subset (\RR^+)^d$ satisfying (\ref{sigmainkp}). Thus taking $t>0$ sufficiently small, (\ref{tree}) will hold for all $x$ with $H_P(x)$ sufficiently large.  \end{proof} 

\begin{remark} \label{32} If $u\in  L_{P,+}$, 
there exists $c$ with $u(y)\geq c+H_P(y)$ on $\CC^d$. Hence
$$\inf_{y\in \CC^d}\{e^{-u(y)}+\frac{1}{t}|y-x|\}\leq \inf_{y\in \CC^d}\{e^{-[c+H_P(y)]}+\frac{1}{t}|y-x|\}\leq e^{-[c+H_P(x)]}$$
which gives 
$$u_t(x)\geq c+H_P(x).$$
Thus $u_t\in  L_{P,+}$
\end{remark}

We use Proposition \ref{ga} in the next sections in proving Theorem \ref{maint}.

\section{Proof of Main Result} Let $P$ be a convex body in $(\RR^+)^d$ satisfying (\ref{sigmainkp2}). As in the case $P=\Sigma$, for $K$ unbounded and $Q$ satisfying (\ref{unbweight}), $V_{P,K,Q}=V_{P,K\cap B_R,Q}$ for $B_R:=\{z: |z|\leq R\}$ with $R$ sufficiently large (cf., \cite{BBL}). Thus in proving Theorem \ref{maint} we may assume $K$ is compact. Theorem 2.10 in \cite{TU} states a Siciak-Zaharjuta theorem for $K,Q$ such that $V_{P,K,Q}$ is continuous (without assuming $Q$ is continuous):

\begin{theorem}\label{SZT} Let $K$ be compact and $Q$ be an admissible weight function on $K$ such that $V_{P,K,Q}$ is continuous. Then $V_{P,K,Q}=\tilde V_{P,K,Q}$ where 
$$\tilde V_{P,K,Q}(z):=\lim_{n\to \infty} [\sup \{\frac{1}{N} \log |p(z)|: p\in Poly(NP), \ ||pe^{-NQ}||_K\leq 1\}]$$
with local uniform convergence in $\CC^d$.

\end{theorem} 

\begin{remark} \label{nm} The fact that the limit $\tilde V_{P,K,Q}(z):=\lim_{n\to \infty} \Phi_n(z)$ exists pointwise follows from the observation that $\Phi_n \cdot \Phi_m \leq \Phi_{n+m}$ (here we are using the notation from (\ref{phin})). Convexity of $P$ is crucial as this property implies that 
$$Poly(nP)\cdot Poly(mP) \subset Poly\bigl((n+m)P\bigr).$$
Note we can also write 
$$\tilde V_{P,K,Q}(z)=\sup\{\frac{1}{deg_P(p)}\log |p(z)|: p \ \hbox{polynomial}, \ ||pe^{-deg_P(p)Q}||_K\leq 1\}.$$
where $deg_P(p)$ is defined in (\ref{degp}) and clearly $V_{P,K,Q}\geq \tilde V_{P,K,Q}$.  
\end{remark}

The proof of Theorem \ref{SZT} is given in \cite{TU}, Theorem 2.10 but the proof is omitted from the final version \cite{Bay}. Here we provide complete details including a proof of the following version of Proposition 2.9 from \cite{Bay}, \cite{TU} which is stated but not proved in these references. Below $dV$ is the standard volume form on $\CC^d$.

\begin{proposition}\label{poly}
Let $P\subset(\Bbb{R}^+)^d$ be a convex polytope and let $f\in \mathcal{O}(\Bbb{C}^d)$ such that
\begin{equation}\label{norm}
\int_{\CC^d}|f(z)|^2e^{-2NH_P(z)}(1+|z|^2)^{-\epsilon}dV(z)<\infty
\end{equation} for some $\epsilon \geq 0$ sufficiently small. Then $f\in Poly(NP).$ 
\end{proposition}
\begin{proof}
Since $P$ is a convex polytope it is given by
$$P=\{x\in \Bbb{R}^d: \ell_j(x)\leq 0,\ j=1,\dots,k\}$$ where 
$$\ell_j(x):=\langle x,r_j\rangle-\alpha_j.$$ Here $r_j=(r_j^1,...,r_j^d)\in \RR^d$ is the primitive outward normal to the $j-$th codimension one face of $P$; $\alpha_j\in \Bbb{R}$; and $<\cdot,\cdot>$ is the standard inner product on $\RR^d$. Recall that the \textit{support function} $h_P:\Bbb{R}^d\to\Bbb{R}$ of $P$ is given by
$$h_P(x)=\sup_{p\in P}\langle x,p\rangle.$$
Fix an index $J\in \Bbb{Z}_+^d\setminus NP$. Replacing $P$ with $NP$ above, this means that
$$h_{NP}(r_j)\leq N\alpha_j<\langle J,r_j\rangle$$ for some $j$. Define $Log:\T\to \Bbb{R}^d$ via
$$Log(z):=(\log|z_1|,\dots,\log|z_d|).$$
The pre-image of $r_j\in \Bbb{R}^d$ under $Log$ is the complex $d$-torus $$\mathcal{S}_{r_j}:=\{|z_1|=e^{r_j^1}\}\times\dots\times\{|z_d|=e^{r_j^d}\}.$$ We conclude that
\begin{equation}\label{ineq}
NH_P(z)=h_{NP}(r_j)<\langle J, Log(z)\rangle
\end{equation} for every $z\in \mathcal{S}_{r_j}$. Clearly, the above inequality is true for every positive multiple of $r_j$ and hence on the set of tori $\mathcal{S}_{tr_j}$ for $t>0$. 

Write $f(z)=\sum_{L\in (\ZZ^+)^d} a_L z^L$. By the Cauchy integral formula
$$a_J=\frac{1}{(2\pi i)^d}\int_{S_{tr_j}}\frac{f(\zeta)}{\zeta^{(J+I)}}d\zeta$$ where
$z^J=z_1^{j_1}\dots z_d^{j_d}$ and $I=(1,\dots,1)$. We want to show that $a_J=0$ for $J\in \Bbb{Z}_+^d\setminus NP$. We write $d|\zeta|=\prod_ie^{tr_j^i}d\theta_i$.
Then by the Cauchy-Schwarz inequality and (\ref{ineq}) we have
\begin{eqnarray*}
|a_J| &\leq&  \frac{1}{(2\pi)^d} (\int_{S_{tr_j}}\frac{|f(\zeta)|^2e^{-2NH_P(\zeta)}}{(1+|\zeta|^2)^{\epsilon}}d|\zeta|)^{\frac12} \big(\int_{S_{tr_j}}\frac{e^{2NH_P(\zeta)}}{|\zeta^{2(J+I)}|}(1+|\zeta|^2)^{\epsilon}d|\zeta|\big)^{\frac12}\\
&\leq&  \frac{1}{(2\pi)^d} (\int_{S_{tr_j}}\frac{|f(\zeta)|^2e^{-2NH_P(\zeta)}}{(1+|\zeta|^2)^{\epsilon}}d|\zeta|)^{\frac12} \big(\int_{S_{tr_j}}\frac{(1+|\zeta|^2)^{\epsilon}}{|\zeta^{2I}|}d|\zeta|\big)^{\frac12}.
\end{eqnarray*}
Thus,
$$\frac{\prod_{i=1}^d\exp(tr_j^i)}{(1+\sum_{i=1}^d\exp(2tr_j^i))^{\epsilon}} |a_J| ^2\leq  (2\pi)^{-d}\int_{S_{tr_j}}|f(\zeta)|^2e^{-2NH_P(\zeta)}(1+|\zeta|^2)^{-\epsilon}d|\zeta|.$$

Note that some components $r_j^i$ of $r_j$ could be negative and some could be nonnegative; e.g., for $\Sigma$ we have $r_j=(0,...0,-1,0,...,0)=-{\bf e}_j$ for $j=1,...,d$ and $r_{d+1}=(1/d,...,1/d)$. Writing $\zeta_i=\rho_ie^{i\theta}$ where $\rho_i:=e^{tr_j^i}$, the above inequality becomes
\begin{equation}\label{integ} \frac{\prod_{i=1}^d\rho_i}{(1+\sum_{i=1}^d\rho_i^2)^{\epsilon}} |a_J| ^2\leq  (2\pi)^{-d}\int_{|\zeta_1|=\rho_i}\cdots \int_{|\zeta_d|=\rho_d}\frac{|f(\zeta)|^2e^{-2NH_P(\zeta)}}{(1+|\zeta|^2)^{\epsilon}}\prod_i\rho_i d\theta_i.\end{equation}
From (\ref{sigmainkp}), $P$ contains a neighborhood of the origin and hence for some $i\in \{1,\dots,d\}$ we have $r_j^i\geq 0$; i.e., we cannot have all $r_j^i<0$. 

{\bf Case 1}: There is some $i\in \{1,...,d\}$ for which $r_j^i>0$. Then for each $i=1,\dots, d$, we integrate both sides of (\ref{integ}) over 
\[\begin{cases}
 1\leq \rho_i \leq T\ \text{if}\ r_j^i\geq 0 \\
 1/T \leq \rho_i \leq 1  \ \text{if}\ r_j^i< 0
 \end{cases}
\] and letting $T\to \infty$ we see that $a_J=0$. 

 {\bf Case 2:} $r_j^i\leq 0$ for every $i=1,\dots,d.$ Note that since $r_j\not=\vec{0}$ there is an $i$ such that $r_j^i<0$.
Since $P$ is a convex polytope this implies that 
$$\int_{\CC^d}|z^J|^2e^{-2NH_P(z)}(1+|z|^2)^{-\epsilon}dV(z)=\infty.$$ On the other hand, since $H_P(z_1,...,z_d)=H_P(|z_1|,...,|z_d|)$, the monomials $a_Lz^L$ occurring in $f$ are orthogonal with respect to the weighted $L^2-$norm in (\ref{norm}). Hence for each such $L$ we have
$$|a_L|^2\int_{\CC^d}|z^L|^2e^{-2NH_P(z)}(1+|z|^2)^{-\epsilon}dV(z)$$
$$\leq \int_{\CC^d}|f(z)|^2e^{-2NH_P(z)}(1+|z|^2)^{-\epsilon}dV(z)<\infty$$
from which we conclude that $a_J=0$.

\end{proof}

\begin{remark} Clearly if $P$ is a convex body in $(\RR^+)^d$ and $f\in \mathcal{O}(\Bbb{C}^d)$ satisfies (\ref{norm}) then for any convex polytope $P'$ containing $P$, $f$ satisfies (\ref{norm}) with $P'$ so that $f\in Poly(NP')$.\end{remark}

We will use the following version of H\"ormander's $L^2$-estimate (\cite[Theorem 6.9]{Dem} on page 379) for a solution of the $\bar \partial$ equation:

\begin{theorem}\label{dbar}
Let $\Omega\subset \CC^d$  be a pseudoconvex open subset and let $\varphi$ be a psh function on $\Omega$. For every $r \in (0,1]$ and every $(0,1)$ form $g=\sum_{j=1}^dg_jd\bar z_j$ with $g_j\in L_{p,q}^2(\Omega,\text{loc}), \ j=1,...,d$ such that $\bar \partial g=0$ and
$$\int_{\Omega}|g|^2e^{-\varphi}(1+|z|^2)dV(z)<\infty$$ 
where $|g|^2:=\sum_{j=1}^d |g_j|^2$ there exists $f\in L^2(\Omega,\text{loc})$ such that $\bar \partial f=g$ and 
$$\int_{\Omega}|f|^2e^{-\varphi}(1+|z|^2)^{-r} dV(z)\leq \frac{4}{r^2}\int_{\Omega}|g|^2e^{-\varphi}(1+|z|^2)dV(z)<\infty.$$
Moreover, we can take $f$ to be smooth if $g$ and $\varphi$ are smooth.
\end{theorem}
 
Finally we will use the following result, which is Lemma 2.2 in \cite{Bay}.
\begin{lemma}\label{helpy}
Let $P$ be a convex body  in $(\RR^+)^d$ and $\psi\in L_{P,+}(\CC^d).$ Then for every $p\in P^{\circ}$ there exists $\kappa,C_{\psi}>0$ such that 
$$\psi(z)\geq \kappa\max_{j=1,\dots,d}\log|z_j|+\log|z^p|-C_{\psi} \ \ \text{for every}\ z\in\T.$$
\end{lemma}

\begin{proof}[Proof of Theorem \ref{SZT}]
From Remark \ref{nm}, given $z_0\in \CC^d$ and $\epsilon >0$ we want to find $N$ large and $p_N\in Poly(NP)$ with 
$$\frac{1}{N} \log |p_N(z)|\leq Q(z), \ z\in K$$
and
$$\frac{1}{N} \log |p_N(z_0)|> V_{P,K,Q}(z_0) -\epsilon.$$

\medskip

\noindent {\bf STEP 1}: We write $V:=V_{P,K,Q}$. Since $V$ is continuous, we can fix $\delta >0$ so that 
$$V(z)>V(z_0)-\epsilon/2 \ \hbox{if} \ z\in B(z_0,\delta).$$
If $V$ is not smooth on $\CC^d$ then we approximate $V$ by smooth psh functions $V_{t}:=\chi_t * V\geq V$ on $\CC^d$ with $\chi$ as in (\ref{chi}). Since $V$ is continuous, $V_{t}$ converges to $V$ locally uniformly as $t\to0.$ 

Let $\eta$ be a test function with compact support in $B(z_0,\delta)=\{z:|z-z_0|<\delta\}$ such that $\eta\equiv1$ on $B(z_0,\frac{\delta}{2}).$ For a fixed point $p\in P^{\circ},$ we define  
$$\psi_{N,t}(z):=(N-\frac{d}{\kappa})(V_t(z)-\frac{\epsilon}{2})+\frac{d}{\kappa}\log|z^p|+d\max_{j=1,\dots,d}\log|z_j-z_{0,j}|$$ where $\kappa>0$ is as in Lemma \ref{helpy} and $\frac{d}{\kappa}\ll N.$ Note that $\psi_{N,t}$ is psh on $\CC^d,$ and smooth away from $z_0$. Applying Theorem \ref{dbar} with the weight function $\psi_{N,t}$, for every $r\in(0,1]$ there exists a smooth function $u_{N,t}$ on $\CC^d$ such that $\bar{\partial}u_{N,t}=\bar{\partial}\eta$ and  
\begin{equation} \label{new43} \int_{\CC^d}|u_{N,t}|^2e^{-2\psi_{N,t}}(1+|z|^2)^{-r}dV(z)\leq \frac{4}{r^2}\int_{\CC^d}|\bar{\partial}\eta|^2e^{-2\psi_{N,t}}(1+|z|^2)dV(z).\end{equation}
Note that the $(0,1)$ form $\bar{\partial}\eta$ is supported in $B(z_0,\delta)\backslash B(z_0,\frac{\delta}{2})$; therefore both integrals are finite. Since $\psi_{N,t}(z)= d\max_{j=1,\dots,d}\log|z_j-z_{0,j}|+0(1)$ as $z\to z_0$ we conclude that $u_{N,t}(z_0)=0.$ Moreover, since $V_t\geq V$ by Lemma \ref{helpy} and (\ref{new43}) we obtain
$$
\int_{\CC^d}|u_{N,t}|^2e^{-2N(V_t-\frac{\epsilon}{2})}(1+|z|^2)^{-r}dV(z)\leq C_1e^{-2N(V(z_0)-\epsilon)}
$$
 where $C_1>0$ does not depend on either $N$ or $t$. 

Next, we let $f_{N,t}:=\eta-u_{N,t}.$ Then $f_{N,t}$ is a holomorphic function on $\CC^d$ such that $f_{N,t}(z_0)=1.$ Furthermore,
$$
\int_{\CC^d}|f_{N,t}|^2e^{-2N(V_t-\frac{\epsilon}{2})}(1+|z|^2)^{-r}dV(z)\leq C_2e^{-2N(V(z_0)-\epsilon)}
$$
and these bounds are uniform as $C_2>0$ is independent of $N$ and $t$. We extract a convergent subsequence $f_{N,t_k}\to f_N$ as $t_k\to 0$ where $f_N$ is a holomorphic function on $\CC^d$ satisfying $f_N(z_0)=1$ and 
\begin{equation}\label{holo}
\int_{\CC^d}|f_N|^2e^{-2N(V-\frac{\epsilon}{2})}(1+|z|^2)^{-r}dV(z)\leq C_2e^{-2N(V(z_0)-\epsilon)}.
\end{equation}
Finally, using $V\in L_{P,+}(\CC^d)$ we see that 
$$\int_{\CC^d}|f_N|^2e^{-2NH_P}(1+|z|^2)^{-r}dV(z)<\infty.$$ Taking $r>0$ sufficiently small, Proposition \ref{poly} implies that $f_N\in Poly(NP)$.

\medskip

\noindent {\bf STEP 2}: We want to modify $f_N\in Poly(NP)$ satisfying (\ref{holo}) and $f_N(z_0)=1$ to get $p_N$. Note $V(z)\leq Q(z)$ on all of $K$. Fix $\rho >0$ and for $r>0$ as above chosen sufficiently small, let 
$$C_r:= \min_{z\in K_{\rho}} (1+|z|^2)^{-r} \ \hbox{where} \ K_{\rho} =\{z: dist(z,K)\leq \rho\}.$$ There exists $\beta =\beta(\rho) >0$ with $|V(z)-V(y)|<\epsilon$ if $y,z\in K_{\rho}$ with $|y-z|<\beta$. Without loss of generality we may assume $\beta \leq \rho$ (or else replace $\beta$ by $\min [\beta, \rho]$). 

For $z\in K$, applying subaveraging to $|f_N|^2$ on $B(z,\beta)\subset K_{\rho}$ we have
$$|f_N(z)|^2\leq C_{\beta} \int_{B(z,\beta)}|f_N(y)|^2 dV(y).$$
Thus, for every $z\in K$
\begin{eqnarray*}
C_r |f_N(z)|^2e^{-2NQ(z)} &\leq& C_r|f_N(z)|^2e^{-2NV(z)}\\
& \leq& C_{\beta} \int_{B(z,\beta)}|f_N(y)|^2e^{-2NV(z)} (1+|y|^2)^{-r} dV(y)\\
&\leq & C_{\beta} \int_{B(z,\beta)}|f_N(y)|^2e^{-2N(V(y)-\epsilon)} (1+|y|^2)^{-r} dV(y)\\
&\leq& C_{\beta} C_2 e^{-2N(V(z_0)-\epsilon)}
\end{eqnarray*}
from (\ref{holo}). Thus taking $p_N:= \sqrt \frac{C_r}{C_{\beta} C_2}e^{N(V(z_0)-\epsilon)}f_N$ we have $p_N\in Poly(NP)$ and
$$\max_{z\in K} |p_N(z)e^{-NQ(z)}|\leq 1.$$
Finally,
$$\frac{1}{N} \log |p_N(z_0)|= V(z_0)-\epsilon + \frac{1}{2N} \log \frac{C_r}{C_{\beta} C_2}.$$
Since none of $C_r, C_{\beta},C_2$ depend on $N$, we have 
$$\frac{1}{N} \log |p_N(z_0)|> V(z_0)-2\epsilon \ \hbox{for} \ N \ \hbox{sufficiently large}.$$ 

This completes the proof of the pointwise convergence of 
$$\frac{1}{N} \log\Phi_N(z):= [\sup \{\frac{1}{N} \log |p(z)|: p\in Poly(NP), \ ||pe^{-NQ}||_K\leq 1\}]$$
to $V_{P,K,Q}(z)$. The local uniform convergence follows as in the proof of Lemma 3.2 of \cite{BS}; this utilizes the  observation that $\Phi_N \cdot \Phi_M \leq \Phi_{N+M}$.

\end{proof}

\begin{remark} \label{obvious} Let $u\in L_{P,+}\cap C(\CC^d)$ with $u \leq Q$ on $K$. The same argument as in Steps 1 and 2 applies to $u$ to show: given $z_0\in \CC^d$ and $\epsilon >0$ we can find $N$ large and $p_N\in Poly(NP)$ with 
$$\frac{1}{N} \log |p_N(z)|\leq Q(z), \ z\in K$$
and
$$\frac{1}{N} \log |p_N(z_0)|> u(z_0) -\epsilon.$$

\end{remark}

Note we have not assumed continuity of $Q$ in Theorem \ref{SZT}. We proceed to do the general case (Theorem \ref{maint}) using Theorem \ref{SZT}; i.e., having proved if $V_{P,K,Q}$ is continuous, then $V_{P,K,Q}=\tilde V_{P,K,Q}$, we verify the equality without this assumption. We begin with an elementary observation.

\begin{lemma} \label{elem} For any $K$ compact and $Q$ admissible,
$$V_{P,K,Q}(z)=\sup \{u(z): u\in L_{P,+}: u\leq Q \ \hbox{on} \ K\}.$$

\end{lemma}

\begin{proof} We have $Q$ is bounded below on $K$; say $Q\geq m$ on $K$. Now $K$ is bounded and so $K\subset D_R=\{z\in \CC^d: |z_j|\leq R, \ j=1,...,d\}$ for all $R$ sufficiently large. Then for $u\in L_P$ with $u\leq Q$ on $K$ we have 
$$\tilde u(z) := \max[u(z), m+ H_P(z/R)]\in L_{P,+}$$
with $\tilde u \leq Q$ on $K$.
\end{proof}

\begin{proof}[Proof of Theorem \ref{maint}]
First we show: {\it if $Q$ is continuous, then $V_{P,K,Q}=\tilde V_{P,K,Q}$.} From Lemma \ref{elem} and Remark \ref{obvious}, it suffices to show that if $u\in L_{P,+}$ with $u\leq Q$ on $K$, given $\epsilon >0$, for $t>0$ sufficiently small, $u_t$ defined in (\ref{ut}) satisfies $u_t \in L_{P,+}$ with $u_t\leq Q+\epsilon$ on $K$. That $u_t\in L_{P,+}$ follows from Proposition \ref{ga} and Remark \ref{32}. Since $u_t\downarrow u$ on $K$, $u_t|_K\in C(K)$, $u$ is usc on $K$, and $K$ is compact, by Dini's theorem, given $\epsilon >0$, there exists $t_0$ such that for all $t<t_0$ we have $u_t\leq Q +\epsilon$ on $K$, as desired. 

Finally, to show $V_{P,K,Q}=\tilde V_{P,K,Q}$ in the general case, i.e., where $Q$ is only lsc and admissible on $K$, we  utilize the argument in \cite{BL}, Lemma 7.3 (mutatis mutandis) to obtain the following.

\begin{proposition} Let $K\subset \CC^d$ be compact and let $w_j=e^{-Q_j}$ be admissible weights on $K$ with $Q_j\uparrow Q$. Then 
$$\lim_{j\to \infty} V_{P,K,Q_j}(z)= V_{P,K,Q}(z) \ \hbox{for all} \ z\in \CC^d.$$
\end{proposition}

Taking $Q_j\in C(K)$ with $Q_j\uparrow Q$, since $V_{P,K,Q_j}=\tilde V_{P,K,Q_j}\leq \tilde V_{P,K,Q}$ for all $j$, we conclude from the proposition that $V_{P,K,Q}=\tilde V_{P,K,Q}$. This concludes the proof of Theorem \ref{maint}. \end{proof}

 We finish this section with some remarks on regularity of $P-$extremal functions. Recall that a compact set $K$ is {\it $L-$regular} if $V_K=V_{\Sigma,K}$ is continuous on $K$ (and hence on $\CC^d$) and $K$ is {\it locally $L-$regular} if it is locally $L-$regular at each point $a\in K$; i.e., if for each $r>0$ the function $V_{K\cap \bar B(a,r)}$ is continuous at $a$ where $\bar B(a,r)=\{z:|z-a|\leq r\}$. For a convex body $P\subset (\RR^+)^d$ we define the analogous notions of {\it $PL-$regular} and {\it locally $PL-$regular} by replacing $V_K$ by $V_{P,K}$. For any such $P$ there exists $A>0$ with $P\subset A\Sigma$; hence
$$V_{P,K}(z)\leq A \cdot V_K(z) \ \hbox{and} \ V_{P,K\cap \bar B(a,r)}(z)\leq A \cdot V_{K\cap \bar B(a,r)}(z)$$
so if $K$ is $L-$regular (resp., locally $L-$regular) then $K$ is $PL-$regular (resp., locally $PL-$regular). Note for $P$ satisfying (\ref{sigmainkp}) there exist $0<a<b<\infty$ with  $a\Sigma \subset P \subset b\Sigma$ so that $K$ is locally $PL-$regular if and only if $K$ is locally $L-$regular.

\begin{corollary} For $K$ compact and locally $L-$regular and $Q$ continuous on $K$, $V_{P,K,Q}$ is continuous.
\end{corollary}

\begin{proof} From Theorem \ref{maint}, $V_{P,K,Q}$ is lowersemicontinuous. We show $V^*_{P,K,Q}\leq Q$ on $K$ from which it follows that $V^*_{P,K,Q}\leq V_{P,K,Q}$ and hence equality holds and $V_{P,K,Q}$ is continuous. 

Since $K$ is locally $L-$regular, it is locally $PL-$regular. Given $a\in K$ and $\epsilon >0$, choose $r>0$ small so that $Q(z)\leq Q(a)+\epsilon$ for $z\in K\cap \bar B(a,r)$. Then
$$V_{P,K,Q}(z)\leq V_{P,K\cap \bar B(a,r), Q(a)+\epsilon}(z)=Q(a)+\epsilon +V_{P,K\cap \bar B(a,r)}(z)$$
for all $z\in \CC^d$. Thus, at $a$, $V_{P,K,Q}(a)\leq Q(a)+\epsilon$. Moreover, by continuity of $V_{P,K\cap \bar B(a,r)}$ at $a$, we have $V_{P,K\cap \bar B(a,r)}(z)\leq \epsilon$ for $z\in \bar B(a,\delta)$, $\delta >0$ sufficiently small. Thus 
$$V_{P,K,Q}^*(a)\leq Q(a)+2\epsilon$$
which holds for all $\epsilon >0$.

\end{proof}

\begin{remark} The converse-type result that for a compact set $K\subset \CC^d$, if $V_{P,K,Q}$ is continuous for every $Q$ continuous on $K$ then $K$ is locally $PL-$regular, follows exactly as in \cite{Dieu} Proposition 6.1. 
\end{remark}

\section{Appendix}

We provide a version of the lemma from Ferrier \cite{F} appropriate for our purposes to show $u_t$ in Proposition \ref{ga} is psh. For $\lambda >0$, we use the distance function $d_{\lambda}:\CC^d \times \CC\to [0,\infty)$ defined as $d_{\lambda}(z,w)=\lambda |z|+|w|$ (in our application, $t=1/\lambda$). 

\begin{lemma} Let $\delta:\CC^d \to [0,\infty)$ be nonnegative. For $\lambda >0$, define
$$\hat \delta_{\lambda}(s):=\inf_{s'\in \CC^d}[\delta(s')+\lambda |s'-s|].$$
Let 
$$\Omega_1:=\{(s,t)\in \CC^d \times \CC: |t|< \delta(s)\}.$$
Then
\begin{equation}\label{trouble}\hat \delta_{\lambda}(s)= d_{\lambda}\bigl( (s,0),(\CC^d \times \CC) \setminus \Omega_1 \bigr).\end{equation}
Furthermore, if $\delta$ is lsc, then $\Omega_1$ is open. Moreover,
$$\Omega_1=\{(s,t): -\log \delta(s) + \log |t|<0\}$$
so that if, in addition, $-\log \delta$ is psh in $\CC^d$, then $\Omega_1$ is pseudoconvex in $\CC^d \times \CC$.  
\end{lemma}     

\begin{proof} This is straightforward; first observe
  $$d_{\lambda}\bigl( (s,0),(\CC^d \times \CC) \setminus \Omega_1 \bigr)=\inf \{\lambda |s-s'|+  |t|:(s',t)\in (\CC^d \times \CC) \setminus \Omega_1 \}$$
  $$=     \inf \{\lambda |s-s'|+  |t|:|t|\geq \delta(s') \} = \inf \{\lambda |s-s'|+  \delta(s'):s'\in \CC^d \}=\hat \delta_{\lambda}(s).$$  
  Next,    
  $$\Omega_1:=\{(s,t)\in \CC^d \times \CC: |t|< \delta(s)\}$$
  $$=\{(s,t)\in \CC^d \times \CC: -\log \delta(s) + \log |t|<0\}.$$
  \end{proof}

\begin{corollary} Under the hypotheses of the lemma, if $-\log \delta$ is psh in $\CC^d$ then $-\log \hat \delta_{\lambda}$ is psh.
\end{corollary}

\begin{proof} Since $\Omega_1$ is pseudoconvex in $\CC^d \times \CC$ and $d_{\lambda}:\CC^d \times \CC\to [0,\infty)$ is a distance function, we have
$$U(s,t):=-\log d\bigl( (s,t),  (\CC^d \times \CC) \setminus \Omega_1 \bigr)$$
is psh. Thus $U(s,0)= -\log \hat \delta_{\lambda}(s)$ is psh.

\end{proof}

                               \end{document}